\documentclass[11pt]{article}
\usepackage[a4paper,margin=26mm]{geometry}
\usepackage[T1]{fontenc}
\usepackage[utf8]{inputenc}
\usepackage{lmodern}
\usepackage{amsmath,amssymb,amsthm,mathtools}
\usepackage{booktabs,longtable,array,tabularx}
\usepackage{enumitem}
\usepackage{xcolor}
\usepackage{hyperref}
\usepackage{tikz}
\usetikzlibrary{arrows.meta,positioning,calc}
\hypersetup{colorlinks=true,linkcolor=blue!70!black,urlcolor=blue!70!black,citecolor=blue!70!black}

\newtheorem{theorem}{Theorem}
\newtheorem{lemma}{Lemma}
\newtheorem{proposition}{Proposition}
\newtheorem{corollary}{Corollary}
\theoremstyle{definition}
\newtheorem{definition}{Definition}

\newcommand{\Cay}{\operatorname{Cay}}

\tikzset{boxnode/.style={draw, rounded corners=2pt, align=center, font=\small, inner sep=5pt},
         labnode/.style={font=\small, align=center},
         flow/.style={-Latex, semithick}}

\title{Hamilton decompositions of the directed 7-torus at odd modulus via root-flat certificates and a prefix-count construction}
\author{Sanghyun Park\\Yonsei University}
\date{May 2026}

\begin{document}
\maketitle

\begin{abstract}
We prove that the directed seven-dimensional equal-side torus
\[
D_7(m)=\Cay\bigl((\mathbb Z_m)^7,\{e_0,e_1,\ldots,e_6\}\bigr)
\]
has a directed Hamilton decomposition for every odd integer \(m\ge3\).  The proof has two
main contributions.  First, we introduce the \emph{root-flat certificate}: a named verification framework in which a Hamilton
decomposition follows from three conditions on a single root flat, namely row Latinness,
layer bijectivity, and primitive return maps.  This abstraction was used informally in the earlier odd
\(D_5(m)\) construction; here it appears as Definition~\ref{def:root-flat-certificate} and Theorem~\ref{thm:certificate-implies}.  Second, for every odd
\(m\ge7\), we give a uniform prefix-coordinate construction: one-layer prefix maps, a symbol-count
criterion, and explicit count matrices produce all seven Hamilton factors without a finite search.
The remaining moduli \(m=3,5\) are exactly the boundary where this prefix-count method cannot
work; they are handled by finite root-flat certificates whose validity is checked in Lean 4.  A Lean 4 formalization verifies the Cayley statement, with the symbolic branch and the finite boundary certificates checked in the same development.
\end{abstract}

\section{Introduction}

A Hamilton decomposition of a finite digraph is a partition of its arc set into directed Hamilton
cycles.  For the equal-side directed torus
\[
D_n(m)=\Cay\bigl((\mathbb Z_m)^n,\{e_0,e_1,\ldots,e_{n-1}\}\bigr),
\]
each vertex has outdegree \(n\).  A Hamilton decomposition, if it exists, consists of exactly
\(n\) arc-disjoint directed Hamilton cycles whose union is the full Cayley arc set.  Hamiltonicity of Cartesian products of directed cycles has a classical history.  For two directed cycles,
Trotter--Erd\H{o}s gave an arithmetic criterion for the existence of a Hamilton cycle, while
Curran--Witte analyzed Hamilton paths in two-factor products and proved that products of three or more
nontrivial directed cycles are Hamiltonian \cite{TrotterErdos,CurranWitte}.  This places the equal-side
directed torus within the standard Hamiltonicity theory of abelian Cayley digraphs; see also the surveys
of Witte--Gallian and Curran--Gallian \cite{WitteGallian,CurranGallian}.

The problem considered here is stronger.  A Hamilton decomposition asks for a partition of the entire
arc set into Hamilton cycles, not merely for one spanning directed cycle.  General Hamilton-decomposition
questions are surveyed by Alspach--Bermond--Sotteau; product decompositions and related directed product
phenomena appear in work of Foregger, Keating, Bogdanowicz, and Darijani--Miraftab--Witte Morris
\cite{AlspachBermondSotteau,Foregger,Keating,Bogdanowicz2017,Bogdanowicz2020,DarijaniMiraftabMorris}.

The directed \(3\)-torus is now known for all \(m\ge3\), by return-map and odometer methods
\cite{ParkD3}; independent constructions and related explorations were also given by Knuth and by
Aquino--Michaels \cite{KnuthClaude,AquinoMichaels}.  The odd \(D_5(m)\) proof used a case-specific zero-set Latin table: one works on a root flat, uses the table for the only nonconstant
layer, certifies the matching property by a finite exact-cover table, and then proves that a normalized
return map is primitive \cite{ParkD5Odd}.  That proof introduced the certificate philosophy, but
it did not isolate the certificate theorem or a uniform large-modulus construction.

The present paper contributes both.  The first contribution is the \emph{root-flat certificate}, introduced here as a named framework in Definition~\ref{def:root-flat-certificate} and Theorem~\ref{thm:certificate-implies}.  It separates local certificate conditions from the lift to the Cayley graph: once row Latinness, layer bijectivity, and primitive return maps hold on one root flat, the full Cayley decomposition follows.  This framework provides a common verification interface for the present construction and for the earlier odd \(D_5(m)\) construction.

The second contribution is the uniform prefix-count construction for dimension \(7\).  For all
odd \(m\ge7\), we construct the seven Hamilton factors by a symbolic prefix-coordinate family and
explicit \(7\times7\) count matrices.  This is a uniform construction for the large-modulus range, in contrast to the case-specific zero-set table used in the odd \(D_5\) proof.  The two remaining odd moduli \(m=3\)
and \(m=5\) are not merely inconvenient exceptions: Proposition~\ref{prop:boundary-obstruction}
shows that the prefix-count method cannot possibly treat them.  We therefore supply finite
root-flat certificates for these two boundary cases, with the full tables stored as ancillary Lean
source artifacts.

\begin{theorem}[Main theorem]\label{thm:main}
For every odd integer \(m\ge3\), the directed torus
\[
D_7(m)=\Cay\bigl((\mathbb Z_m)^7,\{e_0,e_1,\ldots,e_6\}\bigr)
\]
admits a directed Hamilton decomposition.
\end{theorem}

\begin{figure}[t]
\centering
\begin{tikzpicture}[node distance=9mm and 12mm]
\node[boxnode, text width=0.31\textwidth] (u) {\textbf{Uniform branch}\\odd \(m\ge7\)\\Section~\ref{sec:prefix-count}: prefix coordinates,\\one-layer factorization, count matrices};
\node[boxnode, text width=0.31\textwidth, right=22mm of u] (b) {\textbf{Boundary branch}\\\(m=3,5\)\\Section~\ref{sec:exceptional}: finite root-flat certificates\\verified in Lean 4};
\node[boxnode, text width=0.42\textwidth, below=13mm of $(u)!0.5!(b)$] (c) {\textbf{Root-flat certificate theorem}\\Theorem~\ref{thm:certificate-implies}: a valid certificate yields a Hamilton decomposition};
\node[boxnode, text width=0.42\textwidth, below=12mm of c] (main) {\textbf{Main theorem}\\Theorem~\ref{thm:main}: \(D_7(m)\) has a Hamilton decomposition\\for every odd \(m\ge3\)};
\draw[flow] (u.south) -- (c.north west);
\draw[flow] (b.south) -- (c.north east);
\draw[flow] (c.south) -- (main.north);
\end{tikzpicture}
\caption{Logical structure of the proof. The symbolic \(m\ge7\) branch is proved in the paper; the boundary branch \(m=3,5\) is established by finite root-flat certificates checked in Lean 4.}
\label{fig:proof-overview}
\end{figure}
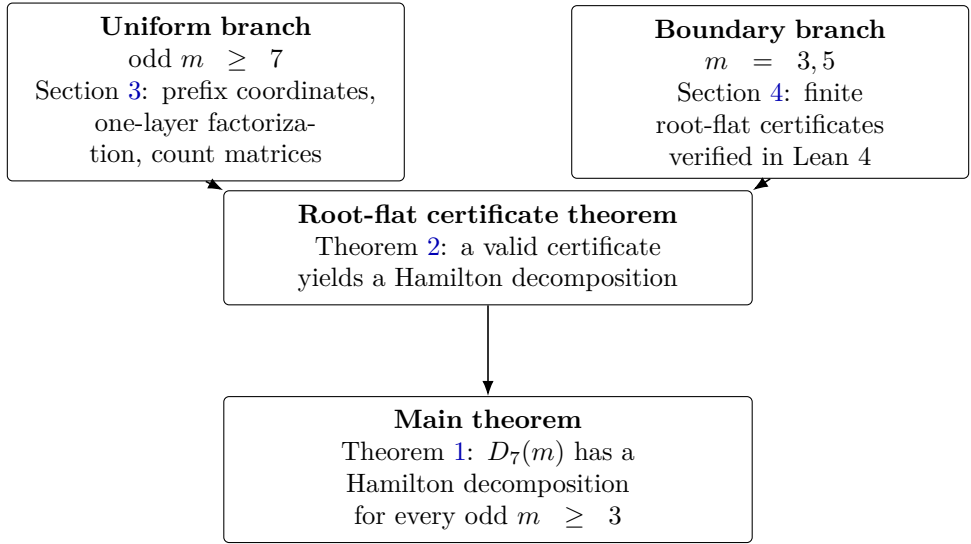

Figure~\ref{fig:proof-overview} summarizes the division between the uniform symbolic branch and the two finite boundary branches.

A separate diagonal Cartesian-power lift gives the multiplicative closure
\[
D_a(m),\ D_b(m^a)\quad\Longrightarrow\quad D_{ab}(m),
\]
for Hamilton decompositions of equal-side directed tori \cite{ParkComposite}.  That mechanism is
useful for composite dimensions but does not produce prime-dimensional base cases such as
\(D_5\) or \(D_7\).  The proof below is therefore a new prime-dimensional base construction rather
than a consequence of the composite lift.

\section{The root-flat certificate framework}\label{sec:root-flat-framework}

Let \(n\ge2\) and \(m\ge2\).  Define the root flat
\[
A_{n,m}=\left\{w=(w_0,\ldots,w_{n-1})\in(\mathbb Z_m)^n:\sum_{i=0}^{n-1}w_i=0\right\}.
\]
For \(0\le i\le n-2\), put
\[
q_i=e_i-e_{n-1},
\]
and put \(q_{n-1}=0\).  Thus \(q_i\in A_{n,m}\), and \(|A_{n,m}|=m^{n-1}\).

On the original torus define the layer function
\[
\sigma(x)=x_0+x_1+\cdots+x_{n-1}\pmod m,
\qquad X_t=\{x:\sigma(x)=t\}.
\]
The map
\[
\iota_t:X_t\longrightarrow A_{n,m},\qquad \iota_t(x)=x-t e_{n-1}
\]
is a bijection.  If \(x\in X_t\), then using generator \(e_i\) sends \(x\) to \(x+e_i\in X_{t+1}\),
and in root-flat coordinates this adds
\[
\iota_{t+1}(x+e_i)-\iota_t(x)=e_i-e_{n-1}=q_i.
\]
For \(i=n-1\), this is \(q_{n-1}=0\).

We now introduce the root-flat certificate.  It is the verification interface used throughout this paper: the local data live on \(A_{n,m}\), while Theorem~\ref{thm:certificate-implies} lifts those data to a Hamilton decomposition of the Cayley digraph.

\begin{definition}[Root-flat certificate]\label{def:root-flat-certificate}
An \((n,m)\)-root-flat certificate consists of functions
\[
d_t:A_{n,m}\times\mathbb Z_n\longrightarrow\mathbb Z_n,
\qquad t\in\mathbb Z_m,
\]
where \(d_t(w,\kappa)\) is the direction used by color \(\kappa\) at layer \(t\) and root-flat point
\(w\).  It is valid if the following hold.
\begin{enumerate}[label=(RF\arabic*)]
\item \textbf{Row Latinness.}  For every \(t\in\mathbb Z_m\) and \(w\in A_{n,m}\), the map
\[
\kappa\longmapsto d_t(w,\kappa)
\]
is a permutation of \(\mathbb Z_n\).
\item \textbf{Layer bijectivity.}  For every \(t\in\mathbb Z_m\) and \(\kappa\in\mathbb Z_n\), the map
\[
P_{t,\kappa}:A_{n,m}\to A_{n,m},
\qquad P_{t,\kappa}(w)=w+q_{d_t(w,\kappa)}
\]
is a bijection.
\item \textbf{Primitive return.}  For every \(\kappa\in\mathbb Z_n\), the return map
\[
R_\kappa=P_{m-1,\kappa}\cdots P_{1,\kappa}P_{0,\kappa}
\]
is a single cycle on \(A_{n,m}\).
\end{enumerate}
\end{definition}

\begin{theorem}[Certificate implies Hamilton decomposition]\label{thm:certificate-implies}
A valid \((n,m)\)-root-flat certificate gives a directed Hamilton decomposition of \(D_n(m)\).
\end{theorem}

\begin{proof}
For color \(\kappa\), define the arc set
\[
E_\kappa=
\left\{\left(x,\ x+e_{d_{\sigma(x)}(\iota_{\sigma(x)}x,\kappa)}\right):x\in(\mathbb Z_m)^n\right\}.
\]
For fixed \(x\), condition (RF1) says that as \(\kappa\) varies, the directions used are all of
\(\mathbb Z_n\).  Therefore the arc sets \(E_\kappa\) partition the outgoing arcs at every vertex,
hence the full Cayley arc set.

Condition (RF2) gives one incoming arc and one outgoing arc at every vertex for each fixed color.
Thus each color class is a directed \(1\)-factor.  Since every arc increases \(\sigma\) by one, a color
orbit returns to the same root flat exactly once every \(m\) steps.  Starting from a different layer cyclically permutes the factors \(P_{t,\kappa}\).  By (RF2), each \(P_{t,\kappa}\) is a bijection, so the resulting cyclic products are conjugate to \(R_\kappa\) and have the same cycle structure.  The cycle structure of a color
class is therefore the cycle structure of the return map \(R_\kappa\), with all lengths multiplied by
\(m\).  By (RF3), \(R_\kappa\) is one cycle of length \(m^{n-1}\).  Hence the color class is one
directed Hamilton cycle of length \(m^n\).  Since this holds for every color, the partition is a
Hamilton decomposition.
\end{proof}

\section{The uniform prefix-count construction for \texorpdfstring{\(m\ge7\)}{m>=7}}\label{sec:prefix-count}

The symbolic branch of the \(D_7\) proof is most transparent in prefix coordinates.  Let
\[
Q_6=(\mathbb Z_m)^6.
\]
For \(0\le r\le6\), define
\[
p_r=(\underbrace{1,\ldots,1}_r,0,\ldots,0)\in Q_6.
\]
We now identify these prefix coordinates with the root flat of Section~\ref{sec:root-flat-framework}.  For
\(w=(w_0,\ldots,w_6)\in A_{7,m}\), define
\[
\Phi(w)_j=\sum_{h=7-j}^{6}w_h,
\qquad 1\le j\le6.
\]
Then \(\Phi:A_{7,m}\to Q_6\) is a bijection.  Its inverse is
\[
w_6=z_1,
\qquad w_i=z_{6-i+1}-z_{6-i}\quad(1\le i\le5),
\qquad w_0=-z_6,
\]
where \(z=\Phi(w)\).  Moreover, for the root-flat directions \(q_i=e_i-e_6\) for \(0\le i\le5\) and \(q_6=0\),
\[
\Phi(w+q_i)=\Phi(w)-p_{6-i}\quad(0\le i\le5),
\qquad
\Phi(w+q_6)=\Phi(w)-p_0.
\]
Thus prefix label \(r\) corresponds to Cayley direction \(e_{6-r}\).  In particular, once a layer map in prefix coordinates sends \(z\) to \(z-p_r\), the corresponding root-flat direction is \(6-r\).

\begin{figure}[t]
\centering
\begin{tikzpicture}[node distance=10mm and 18mm]
\node[boxnode, text width=0.37\textwidth, align=left] (root) {\textbf{Root-flat side}\\\(w=(w_0,\ldots,w_6)\in A_{7,m}\), \(\sum_i w_i=0\).\\[1mm]\(q_i=e_i-e_6\) for \(0\le i\le5\), and \(q_6=0\).};
\node[boxnode, text width=0.40\textwidth, align=left, right=24mm of root] (pref) {\textbf{Prefix-coordinate side}\\\(z=\Phi(w)\in Q_6\), where\\[1mm]\(z_j=\Phi(w)_j=\sum_{h=7-j}^{6} w_h\) for \(1\le j\le6\).\\[1mm]Inverse: \(w_6=z_1\), \(w_i=z_{7-i}-z_{6-i}\) for \(1\le i\le5\), and \(w_0=-z_6\).};
\draw[flow] (root.east) -- node[above, font=\small] {\(\Phi\)} (pref.west);
\node[labnode, text width=0.86\textwidth, align=center, anchor=north] (dict) at ($(root.south)!0.5!(pref.south)+(0,-12mm)$) {\(\Phi(w+q_i)=\Phi(w)-p_{6-i}\) for \(0\le i\le5\), and \(\Phi(w+q_6)=\Phi(w)-p_0\).\\Hence prefix label \(r\) corresponds to Cayley direction \(e_{6-r}\).};
\end{tikzpicture}
\caption{The root-flat/prefix-coordinate dictionary. The bijection \(\Phi\) converts root-flat directions \(q_i\) into prefix displacements \(-p_{6-i}\). This correspondence connects the certificate framework of Section~\ref{sec:root-flat-framework} to the prefix-count construction.}
\label{fig:phi-dictionary}
\end{figure}

\begin{table}[t]
\centering
\small
\begin{tabular}{cccc}
\toprule
Cayley generator & root-flat step & prefix label & prefix displacement\\
\midrule
\(e_0\) & \(q_0=e_0-e_6\) & \(6\) & \(-p_6\)\\
\(e_1\) & \(q_1=e_1-e_6\) & \(5\) & \(-p_5\)\\
\(e_2\) & \(q_2=e_2-e_6\) & \(4\) & \(-p_4\)\\
\(e_3\) & \(q_3=e_3-e_6\) & \(3\) & \(-p_3\)\\
\(e_4\) & \(q_4=e_4-e_6\) & \(2\) & \(-p_2\)\\
\(e_5\) & \(q_5=e_5-e_6\) & \(1\) & \(-p_1\)\\
\(e_6\) & \(q_6=0\) & \(0\) & \(-p_0=0\)\\
\bottomrule
\end{tabular}
\caption{Fixed dictionary between Cayley directions, root-flat directions, and prefix labels.}
\label{tab:prefix-dictionary}
\end{table}

We refer to Figure~\ref{fig:phi-dictionary} and Table~\ref{tab:prefix-dictionary} throughout for the correspondence between Cayley directions and prefix labels.  The symbols used in the construction below form a separate, state-dependent labeling.

\begin{table}[t]
\centering
\small
\begin{tabular}{cc}
\toprule
symbol \(\xi\) & prefix label used at a state \(z\), where \(r=\rho_\tau(z)\)\\
\midrule
\(0\) & \(0\)\\
\(\Delta\) & \(r\)\\
\(a,\;2\le a\le6\) & \(a\) if \(r<a\), and \(a-1\) if \(r\ge a\)\\
\bottomrule
\end{tabular}
\caption{State-dependent symbol-to-prefix-label assignment for the prefix-count construction.}
\label{tab:symbol-label}
\end{table}

No single Cayley direction corresponds to \(\Delta\) uniformly across states; instead, \(\Delta\) denotes the threshold-dependent slot, with label \(r=\rho_\tau(z)\) at the state \(z\).

For a threshold \(\tau\in\mathbb Z_m\) and \(z=(z_1,\ldots,z_6)\in Q_6\), define
\[
\rho_\tau(z)=
\begin{cases}
\min\{i:z_i=\tau\},&\text{if such }i\text{ exists},\\
6,&\text{otherwise}.
\end{cases}
\]
The color variable is denoted by \(\kappa\); the threshold variable is denoted by \(\tau\).  This avoids
using the same letter for two unrelated objects.

We use the symbol set
\[
\Sigma_6=\{0,\Delta,2,3,4,5,6\}.
\]
The symbol \(\Delta\) occupies the position that the integer label \(1\) would occupy among the numeric symbols; we name it separately because the corresponding map \(M_\tau^\Delta\) depends on the threshold \(\tau\) through \(\rho_\tau\), whereas the maps \(M_\tau^a\) for numeric \(a\) have a uniform parametric form.

For \(\xi\in\Sigma_6\), define maps \(M_\tau^\xi:Q_6\to Q_6\) as follows:
\[
M_\tau^0(z)=z,
\]
\[
M_\tau^\Delta(z)=z-p_{\rho_\tau(z)},
\]
and, for the numeric symbols \(2\le a\le6\),
\[
M_\tau^a(z)=
\begin{cases}
z-p_a,&\rho_\tau(z)<a,\\
z-p_{a-1},&\rho_\tau(z)\ge a.
\end{cases}
\]
Thus the superscript is either the formal symbol \(\Delta\) or one of the numeric labels \(\{0,2,3,4,5,6\}\); the
arithmetic in the definition appears only in the numeric cases.

\begin{lemma}[Prefix-coordinate one-layer factorization]\label{lem:prefix-layer}
For every \(m\), every threshold \(\tau\in\mathbb Z_m\), and every \(z\in Q_6\), the seven maps
\[
M_\tau^0,M_\tau^\Delta,M_\tau^2,M_\tau^3,M_\tau^4,M_\tau^5,M_\tau^6
\]
use the seven labels \(0,1,\ldots,6\) exactly once at \(z\), and each map is a bijection of \(Q_6\).
\end{lemma}

\begin{proof}
Let \(r=\rho_\tau(z)\).  The used labels are
\[
0,
\qquad r,
\qquad \{a-1:2\le a\le r\},
\qquad \{a:r<a\le6\},
\]
which are exactly \(0,1,\ldots,6\).

The map \(M_\tau^\Delta\) has inverse
\[
y\longmapsto y+p_{\rho_{\tau-1}(y)}.
\]
Indeed, if \(r=\rho_{\tau-1}(y)\) and \(z=y+p_r\), then \(\rho_\tau(z)=r\).  When \(r=6\), this convention covers both possibilities: first occurrence at coordinate \(6\), and no occurrence at all.  The two cases give the same map because both subtract \(p_6=(1,\ldots,1)\).  Thus \(M_\tau^\Delta(z)=y\).

For \(M_\tau^a\), \(2\le a\le6\), the inverse is also explicit.  Given \(y\), first set
\(z_i=y_i+1\) for \(i<a\).  The predicate ``some coordinate among the first \(a-1\) coordinates of
\(z\) equals \(\tau\)'' is equivalent to ``some coordinate among the first \(a-1\) coordinates of \(y\)
equals \(\tau-1\)''.  If this predicate holds, set \(z_a=y_a+1\); otherwise set \(z_a=y_a\).  All later
coordinates are unchanged.  This recovers a unique preimage, so \(M_\tau^a\) is bijective.
\end{proof}

We now prove the count criterion in a slightly stronger form.  For \(1\le r\le6\), let \(Q_r=(\mathbb Z_m)^r\) and define
\[
\rho^{(r)}_\tau(z)=
\begin{cases}
\min\{i:z_i=\tau\},&\text{if such }i\in\{1,\ldots,r\}\text{ exists},\\
r,&\text{otherwise}.
\end{cases}
\]
In dimension \(r\) we use the symbols
\[
\{0,\Delta,2,\ldots,r\}
\]
with the evident analogues of the maps above.  The induction step also requires one auxiliary symbol \(T\), denoting the full translation
\[
T_r(u)=u-(1,\ldots,1)
\]
on \(Q_r\), because, when projecting from dimension \(r\) to dimension \(r-1\), the top numeric symbol becomes a full translation.

\begin{lemma}[One-cycle skew-product criterion]\label{lem:skew-product}
Let \(X\) be a finite set, \(A:X\to X\) a single cycle of length \(|X|\), and let
\[
F:X\times\mathbb Z_m\to X\times\mathbb Z_m,
\qquad F(x,y)=(A(x),y+\phi(x)).
\]
Put
\[
S=\sum_{x\in X}\phi(x)\in\mathbb Z_m.
\]
Then \(F\) is a single cycle on \(X\times\mathbb Z_m\) if and only if \(S\) generates
\(\mathbb Z_m\), equivalently \(\gcd(S,m)=1\).
\end{lemma}

\begin{proof}
After one full turn of the base cycle, the second coordinate changes by \(S\).  Hence the orbit length
over a base cycle is multiplied by the additive order of \(S\) in \(\mathbb Z_m\).  This order is
\(m/\gcd(S,m)\).  Therefore the total orbit has size \(|X|m\) exactly when \(\gcd(S,m)=1\).
\end{proof}

\begin{lemma}[Prefix-count criterion]\label{lem:count-criterion}
Let \(W=(\xi_0,\ldots,\xi_{m-1})\) be a length-\(m\) word in \(\Sigma_6\).  For thresholds \(\tau_0,\ldots,\tau_{m-1}\), the associated return map is
\[
R_W=M_{\tau_{m-1}}^{\xi_{m-1}}\cdots M_{\tau_1}^{\xi_1}M_{\tau_0}^{\xi_0}.
\]
Let
\[
N_0,N_\Delta,N_2,N_3,N_4,N_5,N_6
\]
be the symbol counts of \(W\).  If
\[
\gcd(N_0,m)=1
\]
and
\[
\gcd(N_k-N_\Delta,m)=1\qquad(2\le k\le6),
\]
then the return map associated to \(W\) is a single \(m^6\)-cycle on \(Q_6\), for any choice of
thresholds in the layers.
\end{lemma}

\begin{proof}
We prove the more general statement in dimension \(r\le6\), allowing the additional full-translation
symbol \(T\).  Let the word length be \(m\).  Let \(N_T\) be the number of occurrences of \(T\).  The
criterion is:
\[
\gcd(m-N_0,m)=1,
\qquad
\gcd(N_k-N_\Delta,m)=1\quad(2\le k\le r).
\]
Since \(m-N_0\equiv -N_0\pmod m\), this is the same first condition as in the statement.  The symbol
\(T\) is counted among the nonzero symbols in \(m-N_0\), but it does not appear in the later
difference conditions.

For \(r=1\), all nonzero symbols, including the auxiliary symbol \(T\) when it occurs, act as the translation \(x\mapsto x-1\) on \(\mathbb Z_m\).  Thus the
return map is translation by minus the number of nonzero symbols, and it is a single cycle exactly when
\(m-N_0\) is a unit modulo \(m\).

Assume the statement known in dimension \(r-1\).  Write
\[
Q_r=Q_{r-1}\times\mathbb Z_m,
\qquad z=(u,y).
\]
Project each layer map to the \(u\)-coordinate.  The symbols \(0,\Delta,2,\ldots,r-1\) project to the
corresponding symbols in dimension \(r-1\).  The top numeric symbol \(r\) projects to the full
translation \(T_{r-1}:u\mapsto u-(1,\ldots,1)\).  The full translation symbol \(T\), if present, also projects to
\(T\).  Therefore the projected word satisfies the induction hypotheses, because the first condition is
unchanged and the difference conditions for \(2\le k\le r-1\) are unchanged.  Hence the projected
return map on \(Q_{r-1}\) is one cycle.

It remains to compute the drift in the last coordinate over one base cycle.  We apply Lemma~\ref{lem:skew-product} with \(X=Q_{r-1}\), with \(A\) equal to the projected return map, and with \(\phi(u)\) equal to the accumulated increment of the last coordinate during one full length-\(m\) word starting from base state \(u\).  The total drift \(S=\sum_u\phi(u)\) is the sum of the following contributions.

Consider a fixed occurrence of \(\Delta\) in the word.  Immediately before this occurrence, the base state is obtained from the initial base state by a bijection of \(Q_{r-1}\).  Hence, as the initial base point ranges over one base cycle, the state seen at this occurrence ranges over all of \(Q_{r-1}\) exactly once.  The \(\Delta\)-map decreases the last coordinate exactly when that state has no coordinate equal to the threshold used in that layer.  This count is independent of the threshold: for any fixed \(\tau\in\mathbb Z_m\), the number of points of \(Q_{r-1}\) with no coordinate equal to \(\tau\) is
\[
(m-1)^{r-1}.
\]
Thus each occurrence of \(\Delta\) contributes \(-(m-1)^{r-1}\) to the total drift.

A fixed occurrence of the top numeric symbol \(r\) similarly sees every base state once.  It decreases the last coordinate exactly when at least one coordinate of the base state equals the current threshold.  Hence each occurrence of \(r\) contributes
\[
-\bigl(m^{r-1}-(m-1)^{r-1}\bigr).
\]
The lower numeric symbols \(2,\ldots,r-1\) do not change the last coordinate.  A full translation symbol \(T\) decreases the last coordinate for every \(u\), hence contributes \(-m^{r-1}\equiv0\pmod m\) to the drift.

Thus the total drift in the last coordinate is congruent modulo \(m\) to
\[
-N_\Delta(m-1)^{r-1}-N_r\bigl(m^{r-1}-(m-1)^{r-1}\bigr)
\equiv (N_r-N_\Delta)(m-1)^{r-1}.
\]
Because \((m-1)^{r-1}\) is a unit modulo \(m\), Lemma~\ref{lem:skew-product} says that the final
fiber is one cycle exactly when \(N_r-N_\Delta\) is a unit modulo \(m\).  This is the required last
condition.  The induction is complete.
\end{proof}

A count matrix has rows indexed by colors and columns ordered as
\[
(0,\Delta,2,3,4,5,6).
\]
If all row sums and column sums are \(m\), form the bipartite multigraph whose left vertices are the seven colors, whose right vertices are the seven symbols, and whose edge multiplicities are the entries of the matrix.  This graph is \(m\)-regular and bipartite.  By Hall's theorem, equivalently by K\H{o}nig's line-coloring theorem \cite{Konig}, it has a perfect matching; removing that matching and repeating gives a decomposition into \(m\) perfect matchings.  Each perfect matching is one layer.

The matrices below were obtained by a constrained search over nonnegative integer entries: row and column sums are fixed at \(m\), the bottom row is constrained to have \(N_\Delta=0\) so that one row absorbs the large \(N_0\) value, and the remaining entries are searched until all unit conditions of Lemma~\ref{lem:count-criterion} hold simultaneously.  The three parametric families reflect the residue class of \(m\) modulo \(6\).  The case \(m=7\) is written separately because the \(6s+1\) family requires \(s\ge2\) for nonnegativity.

\subsection{The matrix for \texorpdfstring{\(m=7\)}{m=7}}

\[
\begin{pmatrix}
1&2&0&0&0&0&4\\
1&2&0&0&0&3&1\\
1&1&0&0&3&2&0\\
1&1&0&3&2&0&0\\
1&1&3&2&0&0&0\\
1&0&2&1&1&1&1\\
1&0&2&1&1&1&1
\end{pmatrix}.
\]
All row and column sums equal \(m=7\).  The values \(N_0\) are all \(1\), hence coprime to \(7\), and each difference \(N_k-N_\Delta\) is a unit modulo \(7\).  Concretely, the first row corresponds to color \(0\): it uses the symbol \(0\) once, the symbol \(\Delta\) twice, no occurrences of symbols \(2\) through \(5\), and the symbol \(6\) four times; its row sum is \(1+2+0+0+0+0+4=7=m\).  A perfect-matching decomposition realizes these row counts as an actual seven-layer schedule \(\mu_t\), with each layer using every symbol exactly once.

\subsection{The matrix for \texorpdfstring{\(m=6s+1\), \(s\ge2\)}{m=6s+1, s>=2}}

\[
\begin{pmatrix}
1&s+1&s-1&s-1&s-1&s-1&s+3\\
1&s+1&s-1&s-1&s-1&s-1&s+3\\
1&s+1&s-1&s-1&s-1&s+2&s\\
1&s&s+1&s+1&s+1&s-1&s-2\\
2&s-1&s&s&s+1&s+1&s-2\\
2&s-1&s+1&s+1&s&s&s-2\\
6s-7&0&2&2&2&1&1
\end{pmatrix}.
\]
All row and column sums equal \(m=6s+1\).  The values \(N_0\) are \(1\), \(2\), and \(m-8\), all coprime to \(m\) since \(m\) is odd and \(\gcd(m-8,m)=\gcd(8,m)=1\).  Each difference \(N_k-N_\Delta\) is in \(\{\pm1,\pm2\}\), hence is coprime to \(m\).

\subsection{The matrix for \texorpdfstring{\(m=6s+3\), \(s\ge1\)}{m=6s+3, s>=1}}

\[
\begin{pmatrix}
1&s+2&s&s&s&s&s\\
1&s+2&s&s&s&s&s\\
1&s+2&s&s&s&s&s\\
1&s-1&s&s&s+1&s+1&s+1\\
2&s-1&s&s&s&s+1&s+1\\
2&s-1&s+1&s+1&s&s&s\\
6s-5&0&2&2&2&1&1
\end{pmatrix}.
\]
All row and column sums equal \(m=6s+3\).  The \(N_0\) values are \(1\), \(2\), and \(m-8\), all coprime to \(m\) since \(m\) is odd and \(\gcd(m-8,m)=\gcd(8,m)=1\).  Each difference \(N_k-N_\Delta\) is in \(\{\pm1,\pm2\}\), hence is coprime to \(m\).

\subsection{The matrix for \texorpdfstring{\(m=6s+5\), \(s\ge1\)}{m=6s+5, s>=1}}

\[
\begin{pmatrix}
1&s+2&s&s&s&s+1&s+1\\
1&s+2&s&s&s&s+1&s+1\\
1&s+2&s&s&s&s+1&s+1\\
1&s&s+1&s+1&s+1&s-1&s+2\\
2&s&s+1&s+1&s+1&s+1&s-1\\
2&s-1&s+1&s+1&s+1&s+1&s\\
6s-3&0&2&2&2&1&1
\end{pmatrix}.
\]
All row and column sums equal \(m=6s+5\).  The \(N_0\) values are \(1\), \(2\), and \(m-8\), all coprime to \(m\) since \(m\) is odd and \(\gcd(m-8,m)=\gcd(8,m)=1\).  Each difference \(N_k-N_\Delta\) is in \(\{\pm1,\pm2\}\), hence is coprime to \(m\).

\begin{proposition}[Uniform prefix-count proof for \(m\ge7\)]\label{prop:mge7}
If \(m\ge7\) is odd, then \(D_7(m)\) has a Hamilton decomposition.
\end{proposition}

\begin{proof}
Use the appropriate count matrix above and decompose its bipartite multigraph into \(m\) perfect matchings.  Write the resulting layer schedule as
\[
\mu_t:\mathbb Z_7\to\Sigma_6,
\qquad t\in\mathbb Z_m,
\]
where \(\mu_t(\kappa)\) is the prefix symbol assigned to color \(\kappa\) in layer \(t\).  Fix, for instance, the threshold \(\tau_t=0\) in every layer; the count criterion allows arbitrary thresholds.

For \(w\in A_{7,m}\), put \(z=\Phi(w)\).  The prefix map \(M_{\tau_t}^{\mu_t(\kappa)}\) sends \(z\) to \(z-p_r\) for a uniquely determined label \(r\in\{0,\ldots,6\}\).  Define the root-flat direction by
\[
d_t(w,\kappa)=6-r.
\]
By the prefix/root-flat dictionary above,
\[
\Phi(w+q_{d_t(w,\kappa)})=M_{\tau_t}^{\mu_t(\kappa)}(\Phi(w)).
\]
Thus the layer maps constructed from the schedule are exactly of the root-flat form required in Definition~\ref{def:root-flat-certificate}.

Lemma~\ref{lem:prefix-layer} proves row Latinness and layer bijectivity.  Lemma~\ref{lem:count-criterion} proves that every color return is a single \(m^6\)-cycle.  The certificate theorem then gives a Hamilton decomposition of \(D_7(m)\).
\end{proof}

\section{The boundary moduli \texorpdfstring{\(m=3,5\)}{m=3,5}}\label{sec:exceptional}

\begin{proposition}[Boundary obstruction for the prefix-count family]\label{prop:boundary-obstruction}
The prefix-count family of Section~\ref{sec:prefix-count} cannot prove \(D_7(3)\) or \(D_7(5)\).
Here a prefix-count schedule means any schedule obtained from the seven-symbol layer factorization of Lemma~\ref{lem:prefix-layer}.  More precisely, when \(m<7\), no seven-color prefix-count schedule can make every return word primitive.
\end{proposition}

\begin{proof}
Fix a color \(\kappa\) and consider its return word in the prefix-count construction.  Project its return map to the first coordinate of \(Q_6\).  The \(0\)-symbol fixes the first coordinate, while every other symbol \(\Delta,2,3,4,5,6\) decreases it by \(1\).  Thus the first-coordinate factor of the return map is the translation
\[
y\longmapsto y-(m-N_0).
\]
If the full return map on \(Q_6\) is a single cycle, its projection to the first coordinate must be transitive on \(\mathbb Z_m\).  Hence this translation must be an \(m\)-cycle, which is equivalent to
\[
\gcd(m-N_0,m)=1,
\]
or equivalently \(\gcd(N_0,m)=1\).  In particular \(N_0\ge1\) for every primitive color.

Seven colors therefore require at least seven occurrences of the \(0\)-symbol across all layers.  But each prefix-count layer uses each symbol exactly once, so across all layers the total number of \(0\)-symbols is exactly \(m\).  If \(m<7\), this is impossible.
\end{proof}

The obstruction is sharp: it identifies exactly which odd moduli lie outside the parameter range of the prefix-count construction.  The point is not that \(D_7(3)\) or \(D_7(5)\) is intrinsically harder as a graph; rather, these moduli lie outside the parameter range of the prefix-count construction.  This motivates the finite certificates below.

\begin{theorem}[Finite boundary certificates]\label{thm:finite-exceptional}
There are valid \((7,3)\)- and \((7,5)\)-root-flat certificates.
\end{theorem}

\begin{proof}[Computer-assisted proof]
The certificates are not printed in this paper.  The boundary return tables have \(729\) and \(15625\) root-flat states respectively, and the full color-layer direction data are too large for inline display in this paper.  The \(m=3\) return table is small enough to print in principle, but we keep both boundary moduli in the same ancillary format so that the paper and the formalization use one certificate interface.  They are
stored as ancillary source files in the Lean formalization repository \cite{TorusProgram}.  The certificates were produced in two stages.  First, a finite exact-cover search over candidate root-flat zero-set selectors identifies selector layers compatible with the row-Latin and layer-bijection conditions.  Second, a return-cycle search over the resulting candidate schedules locates one whose seven color returns are all single cycles.  The proof relies only on the checker, not on the search: given the certificates, the three conditions of Definition~\ref{def:root-flat-certificate} are verified directly.

Concretely, for \(m=3\) the checker enumerates \(A_{7,3}\), verifies the row-Latin and layer-bijection
conditions, and verifies that the seven color returns are single cycles of length \(3^6=729\).  For
\(m=5\), it performs the same check on \(A_{7,5}\), where the return cycles have length
\(5^6=15625\).  The finite certificates are type-checked in the Lean 4 development.  The source bundle records the D7 proof in \texttt{D7Odd/Handoff/CanonicalFamily.lean}, \texttt{D7Odd/Torus.lean}, and \texttt{D7Odd/Cayley.lean}; the top-level Cayley theorem is
\[
\texttt{D7Odd.D7\_odd\_cayley\_unconditional}.
\]
The artifact uses Lean \texttt{v4.30.0-rc2} and mathlib \texttt{v4.30.0-rc2}.  The recorded build command is
\[
\texttt{lake build D5Odd D7Odd}.
\]
The top-level theorem has the interface
\[
\texttt{theorem D7Odd.D7\_odd\_cayley\_unconditional \{m : Nat\} [NeZero m]}
\]
\[
\texttt{(hodd : Odd m) (hm3 : 3 <= m) :}
\]
\[
\texttt{D7Odd.CayleyHamiltonDecompositionD7 m}.
\]
The boundary branch is assembled in the handoff files from the finite certificate objects for \(m=3\) and \(m=5\), including the small root-flat certificates and the corresponding small Hamilton return checks.  The relevant files are
\begin{center}
\begin{tabular}{l}
\texttt{D7Odd/Handoff/SmallBranches.lean},\\
\texttt{D7Odd/Handoff/SmallRank3Certificates.lean},\\
\texttt{D7Odd/Handoff/SmallRank5Certificates.lean}.
\end{tabular}
\end{center}
For the boundary moduli, hand verification of seven \(5^6=15625\)-cycles is impractical; the Lean type checker takes the role of the printed finite certificate.
\end{proof}

\begin{corollary}\label{cor:exceptional}
The directed tori \(D_7(3)\) and \(D_7(5)\) have Hamilton decompositions.
\end{corollary}

\begin{proof}
Apply Theorem~\ref{thm:certificate-implies} to the certificates of
Theorem~\ref{thm:finite-exceptional}.
\end{proof}

\section{Proof of the main theorem}

\begin{proof}[Proof of Theorem~\ref{thm:main}]
Let \(m\ge3\) be odd.  If \(m=3\) or \(m=5\), the result follows from
Corollary~\ref{cor:exceptional}.  If \(m\ge7\), it follows from Proposition~\ref{prop:mge7}.  These cases
exhaust all odd \(m\ge3\).
\end{proof}

\section{Discussion}

The odd \(D_5(m)\) proof and the present \(D_7(m)\) proof share the same root-flat philosophy: local
arc partition and matching are separated from the primitive return-map problem.  The difference is that
for \(D_7\), the large-modulus branch is not another isolated zero-set table.  It is the uniform
prefix-count construction of Section~\ref{sec:prefix-count}.  This gives a uniform proof for all
odd \(m\ge7\), while Proposition~\ref{prop:boundary-obstruction} precisely identifies why the two smaller
odd moduli must be treated separately.

The finite part also changes character from \(D_5\) to \(D_7\).  In the \(D_5\) odd proof, the exceptional
\(m=3\) return cycle has length \(81\), small enough to print as a table.  In \(D_7\), the boundary
returns have lengths \(729\) and \(15625\).  This is the point at which paper-only finite verification
becomes unnatural: the certificate is still finite and elementary, but the finite part is best reviewed together with the ancillary formal artifact.  The Lean type checker plays the same role as the printed \(81\)-cycle table in the
\(D_5\) paper, but at a scale where hand verification of seven \(5^6=15625\)-cycles is no longer reasonable.

The next prime dimension \(D_{11}\) presents two distinct difficulties beyond what the present method handles.  First, a direct zero-set certificate in
\(D_{11}\) would involve exact-cover constraints over subsets of \(\mathbb Z_{11}\), far beyond the
\(D_5\) and \(D_7\) tables.  Second, the primitive return-map proof requires nested first-return sections at multiple scales or a finer rank invariant; a single first-return section of \(D_5\) type no longer suffices.  Composite dimensions remain governed by the diagonal Cartesian-power lift
\cite{ParkComposite}; prime dimensions require independent base constructions.  The root-flat certificate framework gives a common language for such base constructions, even when the selector tables and return-map ranks become more elaborate.

\section*{Acknowledgments and AI disclosure}

This paper was prepared with the help of large language models for exploration and exposition, and a Lean 4 development for formal verification.  Anthropic Claude Opus 4.7
contributed to exposition refinement and to exploratory discussion of the proof structure.  OpenAI GPT-5.5 Pro played a substantial role in the exploration of count-matrix designs, exceptional-certificate construction, proof strategy, and exposition refinement.  OpenAI
GPT-5.5 Codex contributed to the Lean 4 formalization workflow.  The mathematical content of the
paper, including the prefix-count criterion and the validity of the finite certificates, is verified by the
author; the Lean development provides independent machine verification of the finite certificates and of
the final Cayley statement.

\appendix

\section{Ancillary certificate format}\label{app:certificate-format}

For each \(m\in\{3,5\}\), the ancillary certificate contains a table
\[
d_t:A_{7,m}\times\mathbb Z_7\to\mathbb Z_7,
\qquad t\in\mathbb Z_m.
\]
The checker constructs
\[
P_{t,\kappa}(w)=w+q_{d_t(w,\kappa)}
\]
and verifies:
\begin{enumerate}[label=(C\arabic*)]
\item For each \((t,w)\), the list \((d_t(w,0),\ldots,d_t(w,6))\) is a permutation of \(\mathbb Z_7\).
\item For each \((t,\kappa)\), the list \(P_{t,\kappa}(w)\) over all \(w\in A_{7,m}\) is a permutation
of \(A_{7,m}\).
\item For each \(\kappa\), the return permutation
\[
P_{m-1,\kappa}\cdots P_{0,\kappa}
\]
is a single cycle of length \(m^6\).
\end{enumerate}
The Lean files implementing this verification are part of the ancillary source bundle described in
\cite{TorusProgram}.  The repository records Lean \(4.30.0\)-rc2, mathlib \(4.30.0\)-rc2, and the build
command \texttt{lake build D5Odd D7Odd}.

\end{document}